\documentclass{amsproc}
\usepackage{cite}
\usepackage{amsmath,amsfonts,amssymb,amsthm,amscd,mathrsfs,mathdots,bbm,color}
\usepackage{enumerate}

\usepackage{graphicx}
\usepackage[linecolor=white,backgroundcolor=orange,bordercolor=red]{todonotes}

\usepackage{newtxmath}
\usepackage{newtxtext}

\usepackage{ stmaryrd }

\newcommand{\R}{\mathbb{R}}
\newcommand{\C}{\mathbb{C}}
\newcommand{\N}{\mathbb{N}}

\newcommand{\EE}{\mathbb{E}}

\newcommand{\M}{\mathbb{M}}

\newcommand{\de}{ \mathrm{d}}

\newcommand{\be}{\begin{equation}}
\newcommand{\ee}{\end{equation}}

\newcommand{\ben}{\begin{equation*}}
\newcommand{\een}{\end{equation*}}

\newtheorem{thm}{Theorem}

\newtheorem{cor}{Corollary}

\newtheorem{lem}{Lemma}

\theoremstyle{definition}

\theoremstyle{remark}
\newtheorem{rem}{Remark}

\newcommand{\arXiv}{arXiv:}

\title[Moments of discrete orthogonal polynomial ensembles]{Moments of discrete orthogonal polynomial ensembles
}

\author[P. Cohen]{Philip Cohen}
\address{School of Mathematics and Statistics, University College Dublin, Dublin 4, Ireland}
\author[F. D. Cunden]{Fabio Deelan Cunden}
\author[N. O'Connell]{Neil O'Connell}
\thanks{Research supported by ERC Advanced Grant 669306. Research of FDC partially supported by Gruppo Nazionale di Fisica Matematica  GNFM-INdAM.}

\begin{document}

\begin{abstract}  We obtain factorial moment identities for the Charlier, Meixner and Krawtchouk orthogonal polynomial ensembles. Building on earlier results by Ledoux [Elect. J. Probab. 10, (2005)], we find hypergeometric representations for the factorial moments when the reference measure is Poisson (Charlier ensemble) and geometric (a particular case of the Meixner ensemble). In these cases, if the number of particles is suitably randomised, the factorial moments have a polynomial property, and satisfy three-term recurrence relations and differential equations. In particular, the normalised factorial moments of the randomised ensembles are precisely related to
the moments of the corresponding equilibrium measures.  We also briefly outline how these results can be interpreted as 
Cauchy-type identities for certain Schur measures.
\end{abstract}

\maketitle

\section{Introduction and definitions}
An orthogonal polynomial ensemble is a probability measure on $\R^N$ given by
\be
\de Q(x)=\frac{1}{Z_N}\Delta(x)^2\prod_{j=1}^N\de\mu(x_j),
\label{eq:dQ}
\ee
where $x = (x_1, \dots, x_N)$, $\mu$ is a probability measure on the real line having all moments,
\be
\Delta(x)=\det_{1\leq i,j\leq N}\left(x_i^{j-1}\right)=\prod_{i<j}(x_j-x_i),
\ee 
and $Z_N$ is a normalisation constant. Orthogonal polynomial ensembles arise as joint eigenvalue distributions of certain unitarily invariant models of random matrices~\cite{Mehta67}. 
Dyson~\cite{Dyson62} showed that the measure~\eqref{eq:dQ}  also has a natural interpretation as a $2d$ Coulomb gas on the line (also known as log-gas). 
 There are a number of models from statistical physics, probability theory and combinatorics, which are described in terms of orthogonal polynomial ensembles. The reader can consult the survey article by K\"onig~\cite{Konig05}.
 
The measure~\eqref{eq:dQ} can be conveniently analysed by using the so-called \emph{orthogonal polynomial method} pioneered by Mehta~\cite{Mehta60} in the study of random matrices. Denote by $p_n(x)$, $n\in\N$, the orthonormal polynomials with respect to the measure $\mu$, i.e. $\int p_m(x)p_n(x)\de \mu(x)=\delta_{mn}$. Then, a standard calculation shows that
\be
\de Q(x)=\frac{1}{N!}\det_{1\leq i,j\leq N}K(x_i,x_j)\prod_{j=1}^N\de\mu(x_j),
\label{eq:dQ_K}
\ee
where $K(x,y)=\sum_{n=0}^{N-1}p_n(x)p_n(y)$. In fact, all the marginals~\eqref{eq:dQ_K} can be expressed as determinants of  the correlation kernel $K(x,y)$. In particular, for suitable functions $f$,
\be
\int_{\R^N}\left(\frac{1}{N}\sum_{n=1}^Nf(x_n)\right)\de Q(x_1,\dots,x_N)=\int_{\R}f(x)\de\rho_N(x),
\ee
where the normalised one-point function $\rho_N$ is the probability measure
\be
\de\rho_N(x)=\frac{1}{N}\sum_{n=0}^{N-1}p_n(x)^2\de\mu(x).
\ee
Therefore the study of the Coulomb gas $Q$ for some reference measure $\mu$ amounts to understanding properties of the associated orthogonal polynomials $p_n(x)$.
\par 
Certainly, one of the most studied quantities on orthogonal polynomial ensembles are the so-called \emph{moments}. Let $X=(X_1,\dots,X_N)$ be distributed according to the probability measure $Q$. Then, the $k$th moment of $X$ is the expectation of the $k$th power sum
\be
\EE_Q\frac{1}{N} \sum_{n=1}^NX_n^k=\int_{\R}x^k\de\rho_N(x),\quad k\in\N.
\label{eq:mom}
\ee
It is well-known that moments of orthogonal polynomial ensembles admit a $1/N$-expansion as $N\to\infty$, whose coefficients satisfy a hierarchy of conditions known as \emph{loop equations} initially derived in theoretical physics~\cite{ Ambjorn90,Ambjorn93}, and then proved using several analytical techniques~\cite{Albeverio01, Ercolani03,Haagerup12,Forrester17,Witte14,Dubrovin16,Kopelevitch18}. (These expansions are special instances of the asymptotic series and associated  topological recursions for more general probability models called $\beta$-ensembles; see~\cite{BG,CE,Eynard15,Eynard16,Guionnet15} for the state of the art.)

When the reference measure $\mu$ is the standard Gaussian, Gamma, or Beta distribution, $Q$ is the eigenvalues distribution of the Gaussian (GUE), Laguerre (LUE) or Jacobi (JUE) unitary ensembles of random matrices, with associated Hermite, Laguerre and Jacobi polynomials~\cite{Mehta67}. These classical ensembles are exactly solvable, in the sense that the coefficients of the $1/N$-expansion have an explicit description in terms of enumeration of maps (or factorisations in the symmetric group) on surfaces labelled by their genus~\cite{Hanlon92,Graczyk03,Collins03,Collins14,Mingo17,Cunden19b}. The $1/N$-expansion for the GUE, LUE and JUE is actually convergent; moments~\eqref{eq:mom} of the classical random matrix ensembles are polynomials or rational functions in the variable $N$ (the size of the matrix). 

Moreover, the moments~\eqref{eq:mom} of Gaussian, Laguerre and Jacobi unitary ensembles satisfy three-term recurrences in $k$ (the order of the moment). These remarkable recursions were first discovered by Harer and Zagier~\cite{Harer86} for the GUE, and by Haagerup and Thorbj\o rnsen for the LUE~\cite{Haagerup03} (the extension of the Haagerup-Thorbj\o rnsen recursion to moments of the inverse LUE was examined later in~\cite{Cunden16b}). Recursions in $k$ for moments of the JUE were obtained by Ledoux~\cite{Ledoux04}, and recast as three-term recurrences in~\cite{Cunden19}. 

Recent studies by Cunden, Mezzadri, O'Connell and Simm~\cite{Cunden19} clarified the origin of these three-term recurrences in $k$ for moments of the classical ensembles. Those recursions originate from the second order differential equations (continuous Sturm-Liouville problems) satisfied by the Hermite, Laguerre and Jacobi polynomials. In fact, the Harer-Zagier, Haagerup-Thorbj\o rnsen and Ledoux recursions can be interpreted as second order difference equation in the variable $k$ (discrete Sturm Liouville problems). It turns out that their solutions are hypergeometric orthogonal polynomials in $k$ belonging to the Askey scheme~\cite{Koekoek10}. The polynomial structure, the orthogonality relation and the hypergeometric representation of the moments as function of $k$ explain several nontrivial symmetries and provide new results. For instance, by duality, moments of classical classical random matrices, if suitably normalised can be view as hypergeometric  polynomials in the variable $(N-1)$ as well; therefore, they also satisfy three-term recurrences in $(N-1)$, whose nature is different from the general topological recursion. See~\cite{Cunden19}  for the details.  

This link between moments of the classical continuous orthogonal polynomial (OP) ensembles and hypergeometric OPs suggests that other related `solvable models' might exhibit unexpected representation of the moments in terms of hypergeometric series and/or special orthogonal polynomials.

In this paper we examine the classical discrete OP ensembles. Namely, the Charlier, Meixner, and Krawtchouk ensembles, which correspond to the Poisson, negative binomial, and binomial reference measures $\mu$ on $\N$, respectively.  In these ensembles the reference measure $\mu$ is discrete. 
Hence, the measure $Q$ can be thought of as a discrete Coulomb gas: a system of $N$ particles on the integer lattice $\N$, distributed according to a common distribution $\mu$ under the influence of the repelling density  $\Delta(x)^2$.
It turns out (as is often the case for integer valued random variables) that rather than the moments~\eqref{eq:mom}, the \emph{factorial moments} 
\be
\EE_Q\frac{1}{N}  \sum_{n=1}^NX_n(X_n-1)\cdots(X_n-k+1)=\int_{\R}x(x-1)\cdots(x-k+1)\de\rho_N(x),\quad k\in\N,
\label{eq:facmom}
\ee
are more natural to look at in the discrete setting. 
\par
The outline of the paper is as follows. 
In Section~\ref{sec:Led} we introduce the OP ensembles considered in this paper. Then, we give the factorial moment formulae for the Charlier and Meixner ensembles calculated by Ledoux~\cite{Ledoux05}, and we calculate the corresponding formula for the Krawtchouk ensemble, whose statement is omitted in~\cite{Ledoux05}.

In Section~\ref{sec:hyper}, motivated by recent developments~\cite{Cunden19} on moments of the classical continuous OP ensembles, we give hypergeometric representations for the factorial moments of the Charlier ensemble, and the special case of Meixner ensemble with $\gamma=1$. From these formulae it is clear that the $k$th factorial moment (suitably normalised) is a polynomial in~$k$. 

The main novel observation reported in the paper is contained in Section~\ref{sec:randomised}. It turns out that in the Charlier, and $\gamma=1$ Meixner ensembles, if the number of particles $N$ is randomised according to certain distributions, then the $k$th factorial moment has an even simpler hypergeometric representation and can be expressed as a Jacobi or Legendre polynomial of degree $k$. It follows that the randomised factorial moments satisfy three-term recursions in $k$ (similar to the Harer-Zagier, Haagerup-Thorbj\o rnsen and Ledoux recurrences in the continuous ensembles), and second order differential equations (Sturm-Liouville problems). 

In the last section of this work, we highlight a nice feature of the normalised factorial moments for the randomised 
Charlier and $\gamma=1$ Meixner ensembles, 
namely that they are precisely related to the moments of the corresponding equilibrium measures.

\noindent {\em Acknowledgement.}  Thanks to the referee for careful
reading of an earlier version and many helpful suggestions.

\section{Moment formulae for Charlier, Meixner and Krawtchouk ensembles} \label{sec:Led}
We give first the precise definitions of the three classical OP ensembles considered here. These ensembles
arise in connection with random partitions~\cite{Johansson99,Johansson00} and 
conditioned random walks~\cite{Konig05,Konig02,OC03}.

In the following we will use the hypergeometric notation
\be
 {}_pF_q \left( \begin{matrix} a_1, \dots , a_p \\ b_1 , \dots , b_q \end{matrix} ; z \right)=\sum_{i=0}^{\infty}\frac{(a_1)_i\cdots (a_p)_i}{(b_1)_i\cdots (b_q)_i}\frac{z^i}{i!}.
 \label{eq:hyp}
\ee
where $(a)_i$ denotes the rising factorial $a(a+1) \dots (a+i-1)$. If one of the parameters $a_1,\cdots,a_p$ is a nonpositive integer, then the series terminates. 
Recall that a series $ \sum_{i=0}^{\infty} t_i$ is hypergeometric if the ratio of consecutive terms $\frac{t_{i+1}}{t_i}$ is a rational function of $i$, and from~\eqref{eq:hyp} we see that
	\begin{align} \label{hyperg}
	\frac{t_{i+1}}{t_i} = \frac{(i+a_1)\dots (i+a_p)}{(i+b_1) \dots (i+b_q)(i+1)} z 
	\implies \sum_{i=0}^{\infty} t_i = {}_pF_q \left( \begin{matrix} a_1, \dots , a_p \\ b_1 , \dots , b_q \end{matrix} ; z \right) t_0.
	\end{align}

\begin{enumerate}[(a)]
\item The Charlier ensemble corresponds to a Poisson reference measure $\mu=\mu^{\theta}$ with rate $\theta>0$, given by\footnote{In this paper $0\in\N$.}
\be
\mu^{\theta}(x) = \frac{\theta^x e^{-\theta}}{x!},\quad x \in \mathbb{N}.
\label{eq:Poisson}
\ee
The $n$th normalised Charlier polynomial $c_n(x;\theta)$ can be written as~\cite[9.14]{Koekoek10}
\be
c_n(x;\theta) = \sqrt{\frac{\theta^n}{n!}} {}_2F_0 \left( \begin{matrix} -n , -x  \\ - \end{matrix} ; - \frac{1}{\theta}\right).
 \label{eq:charlier_pol}
\ee
The Charlier one-point function will be denoted by
\be
\rho_{N}^{\theta}(x) = \frac{1}{N} \sum_{n=0}^{N-1} c_n(x; \theta)^2 \mu^{\theta}(x) .
\ee
\item The Meixner ensemble corresponds to a negative binomial reference measure $\mu=\mu_q^{\gamma}$ with parameters $0<q<1$ and $\gamma>0$, where $(\gamma)_x$ denotes the rising factorial as defined above,
\be
\mu_q^{\gamma}(x) = \frac{(\gamma)_x}{x!} q^x (1-q)^{\gamma},\quad x \in \mathbb{N}.
\label{eq:neg_bin}
\ee
The $n$th normalised Meixner polynomial $m_n(x; \gamma, q)$ is~\cite[9.10]{Koekoek10}
\be
 m_n(x; \gamma, q) = \sqrt{\frac{(\gamma)_n q^n}{n!}} {}_2 F_1 \left( \begin{matrix} -n, -x  \\ \gamma \end{matrix} ; 1-\frac{1}{q} \right). 
 \label{eq:meixner_pol}
\ee
We shall denote the Meixner one-point function by
\be
\rho_{N,q}^{\gamma} (x) = \frac{1}{N} \sum_{n=0}^{N-1} m_n(x; \gamma, q)^2 \mu_q^{\gamma} (x).
\ee
\item The Krawtchouk ensemble corresponds to the binomial distribution $\mu=\mu_p^K$ with $K\geq N$ trials and success probability $0<p<1$,
\be
\mu_{K,p}(x) = \binom{K}{x}p^x (1-p)^{K-x},\quad x=0,1,2,\dots,K.
\label{eq:bin}
\ee
The $n$th normalised Krawtchouk polynomial $ k_n(x;p,K)$ has the hypergeometric representation~\cite[9.11]{Koekoek10}
\be
k_n(x;p,K) = \sqrt{\binom{K}{n} \left(\frac{p}{1-p}\right)^n} {}_2 F_1 \left( \begin{matrix} -n , -x  \\ -K \end{matrix} ; \frac{1}{p} \right), \quad n=0,1,2,\dots,K.
 \label{eq:krawtchouk_pol}
\ee
Note that, unlike in the Charlier and Meixner cases, this is a finite family of \hbox{$K+1$} polynomials.
The binomial distribution~\eqref{eq:bin} corresponds to a negative binomial distribution with negative parameters~\eqref{eq:neg_bin}
\be
\mu_{K,p}(x)=\mu_{-p/(1-p)}^{-K}(x),
\ee
and the Krawtchouk polynomials are related to the Meixner polynomials given by~\eqref{eq:meixner_pol} in the following way:
\be
k_n(x;p,K)=m_n\left(x;-K,-p/(1-p)\right).
\ee
We denote the one-point function of the Krawtchouk ensemble as
\be
\rho_{N,K,p} (x) = \frac{1}{N} \sum_{n=0}^{N-1} k_n(x;p,K)^2 \mu_{K,p} (x).
\ee
\end{enumerate}
\par

In~\cite{Ledoux05}, formulae for the factorial moments of the Charlier and Meixner ensembles are calculated. We state the results here.
\par
\begin{thm}[Ledoux~\cite{Ledoux05}] 
\label{prop:Charlier}
	Let $\rho_{N,{\theta}}(x)$ denote the Charlier one-point function. Then, the factorial moment 
	\be
M^{\theta}(k,N)=\int x(x-1)\cdots(x-k+1) \mathrm{d} \rho_{N,{\theta}}(x)
	\label{eq:Charlier_mom_def}
	\ee
	is given by
\be
M^{\theta}(k,N)
	= \sum_{i=0}^k \theta^{k-i} \binom{k}{i}^2 \frac{1}{N} \sum_{l=i}^{N-1} \frac{l!}{(l-i)!} .
	\label{eq:Charlier_mom}
	\ee
\end{thm}

\begin{thm}[Ledoux~\cite{Ledoux05}]
 \label{prop:Meixner}

    Let $\rho_{N,q}^{\gamma} (x)$ denote the Meixner one-point function. Then, the factorial moment 
	\be
M^{\gamma}_{q}(k,N)=\int  x(x-1)\cdots(x-k+1) \mathrm{d} \rho_{N,q}^{\gamma} (x)
	\label{eq:Meixner_mom_def}
	\ee
	is given by
	\be
M^{\gamma}_{q}(k,N)=\sum_{i=0}^k  \left( \frac{q}{1-q} \right)^k q^{-i} \binom{k}{i}^2 \frac{1}{N} \sum_{l=i}^{N-1} \frac{l! }{(l-i)!}(\gamma+l)_{k-i}.
	\label{eq:Meixner_mom}	
	\ee
\end{thm}

We can also compute the factorial moments of the Krawtchouk ensemble. 
\begin{thm}
\label{prop:Krawtchouk}
	Let $\rho_{N,K,p} (x)$ denote the Krawtchouk  one-point function. Then, the factorial moment
	\be
M_{K,p}(k,N)=\int x(x-1)\cdots(x-k+1)  \de\rho_{N,K,p} (x)
	\label{eq:Krawtchouk_mom_def}
	\ee
	is given by
\be
M_{K,p}(k,N)
	= \sum_{i=0}^k p^{k-i}\left(1-p \right)^i \binom{k}{i}^2 \frac{1}{N} \sum_{l=i}^{N-1} \frac{l! (K-l)!}{(l-i)!(K-l-k+i)!} .
		\label{eq:Krawtchouk_mom}
\ee
\end{thm}

\begin{proof}
From the general setting of~\cite{Ledoux05}, we deduce the following `integration by parts' formula
\be
 \int xf(x) \mathrm{d} \mu_{K,p}(x) = pK \int f(x+1) \mathrm{d} \mu_{K-1,p}(x).
 \ee
 Hence, we have
 \begin{multline}
\int x(x-1)\cdots(x-k+1) k_l(x;p,K)^2 \mathrm{d}\mu_{K,p}(x) \\ 
= p^k \frac{K!}{(K-k)!} \int k_l(x+k;p,K)^2 \mathrm{d} \mu_{K-k,p}(x).
\end{multline}	
In order to calculate this integral, we use the forward shift operator~\cite[Eq. (9.11.6)]{Koekoek10}
\be 
k_l(x+1;p,K) - k_l(x;p,K) = \sqrt{\frac{l}{Kp(1-p)}}k_{l-1}(x;p,K-1),
\ee
and iterate so that
	\begin{align*}
	k_l(x+k;p,K) &= k_l(x+k-1;p,K) - \sqrt{\frac{l}{Kp(1-p)}} k_{l-1}(x+k-1;p,K-1) \\
	&= \dots = \sum_{i=0}^{k \wedge l} (-1)^i \binom{k}{i} \sqrt{\frac{(-l)_i}{(-K)_i p^i (1-p)^i}} k_{l-i}(x;p,K-i).
	\end{align*}
	Hence we observe that
	\begin{multline} 
	k_l(x+k;p,K)^2 = \\ \sum_{i,j=0}^{k \wedge l} (-1)^{i+j} \binom{k}{i} \binom{k}{j}  \sqrt{\frac{(-l)_i (-l)_j}{(-K)_i (-K)_j}} (p(1-p))^{-\frac{i+j}{2}} k_{l-i}(x;p,K-i) k_{l-j}(x;p,K-j). 
	\label{eq:Kraw_step}
	\end{multline}
	We now make use of the Krawtchouk generating function~\cite[Eq. (9.11.11)]{Koekoek10}
\be
\left( 1- \frac{1-p}{p}t \right) ^x (1+t)^{K-x} = \sum_{n=0}^K \sqrt{\binom{K}{n}\left(\frac{1-p}{p} \right)^n} k_n(x;p,K)t^n. 
\label{eq:gen_f_Kraw}
\ee	
	So, to calculate the integral $ \int  k_{l-i}(x;p,K-i) k_{l-j}(x;p,K-j)   \mathrm{d} \mu_{K-k,p}(x) $ we find the $t^{l-i} s^{l-j} $ coefficient of	
\be
	\int \sum_{n=0}^{K-i} \sum_{m=0}^{K-j} \sqrt{\binom{K-i}{n} \binom{K-j}{m} \left( \frac{1-p}{p} \right)^{n+m}}k_{n}(x;p,K-i) k_{m}(x;p,K-j)  t^n s^m \mathrm{d} \mu_{K-k,p}(x) 
\label{eq:gen_step}
\ee	
which, by the generating function~\eqref{eq:gen_f_Kraw}, is equal to	
	\begin{multline*}
	    \int \left( 1- \frac{1-p}{p}t \right) ^x (1+t)^{K-x-i} \left( 1- \frac{1-p}{p}s \right) ^x (1+s)^{K-x-j}  \mathrm{d} \mu_{K-k,p}(x) \\
		= (1+t)^{k-i} (1+s)^{k-j} \left( 1+\frac{1-p}{p}st \right) ^{K-k} ,
		\end{multline*}	
	since, after taking out a factor of $(1+t)^{k-i}(1+s)^{k-j}$, we find that the remaining discrete integral can be simplified using the binomial theorem. Thus for any integers $n$ and $m$, we can calculate the $t^n s^m$ coefficient by counting the terms in this product with the corresponding coefficients, to get
\[
\sum_{a=0}^{k-i} \binom{k-i}{a} \binom{K-k}{n-a} \binom{k-j}{m-n+a} \left(\frac{1-p}{p} \right)^{n-a}.  
\]
	Hence taking $n=l-i$, $m=l-j$ and relabelling $a=r-i$, the $t^{l-i} s^{l-j}$ coefficient of~\eqref{eq:gen_step} gives
\begin{multline}
\int  k_{l-i}(x;p,K-i) k_{l-j}(x;p,K-j)   \mathrm{d} \mu_{K-k,p}(x) =\\
\sum_{r=i \vee j}^{k \wedge l} \binom{k-i}{r-i} \binom{k-j}{r-j} \binom{K-k}{l-r}  \left(\frac{1-p}{p} \right)^{l-r} .
\end{multline}
	Combining this with~\eqref{eq:Kraw_step} we have a formula for 
	$\int k_l(x+k;p,K)^2 \mathrm{d} \mu_{K-k,p}(x)$. Finally, taking the sum over $l$ from $0$ to $N-1$, this is
$$
	\frac{p^k}{N} \sum_{l=0}^{N-1} \sum_{r=0}^{k \wedge l} \left(\frac{1-p}{p} \right)^r \binom{k}{r}^2 \frac{l! (K-l)!}{(l-r)!(K-l-k+r)!}   $$
	which, once we swap the order of the sums, is the desired formula.
\end{proof} 

\begin{rem}
The reader can check that the moments~\eqref{eq:Krawtchouk_mom} of the Krawtchouk ensemble can be obtained formally from the Meixner case~\eqref{eq:Meixner_mom} by the substitution $q \to -p/(1-p)$ and $\gamma \to -K$:
\be
M_{K,p}(k,N)=M^{-K}_{-\frac{p}{1-p}}(k,N).
\ee
\end{rem}

\section{Hypergeometric moment formulae}
\label{sec:hyper}

\subsection{Hypergeometric representations}
Theorems~\ref{prop:Charlier}, \ref{prop:Meixner} and \ref{prop:Krawtchouk} are the starting points of our analysis on the factorial moments of the classical discrete OP ensembles as functions of $k$. Specifically we look for properties that mirror the nice structure recently found on the classical continuous OP ensembles~\cite{Cunden19}. A `nice property' in the continuous setting is that the moments have a hypergeometric representation. We found similar hypergeometric series for the factorial moments of the classical discrete OP ensembles when the reference measure is Poisson $\mu^{\theta}(x)=\theta^xe^{-\theta}/x!$, $\theta>0$ (Charlier ensemble) and geometric $\mu_q^{1}(x) =  q^x (1-q)$ (Meixner ensemble with parameter $\gamma=1$).

We will make use of the following identity, which can be proved by induction on $N$:
\begin{align} 
\label{eq:pochsum2}
\sum_{l=i}^{N-1} \frac{(l+k-i)!}{(l-i)!} &= \frac{(N+k-i)!}{(k+1)(N-i-1)!} = \frac{1}{k+1}(N-i)_{k+1}.
\end{align}
\begin{thm}[Charlier factorial moment] \label{thm:Charlier}
The factorial moment of the Charlier ensemble can be written as
\be
M^{\theta}(k,N)= \theta^k {}_3F_1 \left( \begin{matrix} -k , -k , 1-N \\ 2 \end{matrix} ; - \frac{1}{\theta} \right).
\label{eq:Charl_fact}
\ee
In particular, $\theta^{-k}M^{\theta}(k,N)$ can be extended to the entire complex plane to a polynomial in $k$ of degree $2(N-1)$, and when $k \in \mathbb{N}$, it is a polynomial in $(N-1)$ of degree $k$.
\end{thm}
\begin{proof}
	Firstly, by \eqref{eq:pochsum2} (with $k=i$) and the identity $\frac{1}{N} (N-i)_{i+1} = (N-i)_{i}$, we simplify formula~\eqref{eq:Charlier_mom} as 
\be
M^{\theta}(k,N)=\sum_{i=0}^k \theta^{k-i} \binom{k}{i}^2 \frac{(N-i)_{i}}{(i+1)}. 
\ee
If we denote 
$$
t_i=\theta^{k-i} \binom{k}{i}^2 \frac{(N-i)_{i}}{(i+1)},
$$
we have $t_0 = \theta^k$, and for all $0\leq i \leq k$,
$$
\frac{t_{i+1}}{t_i} = \frac{(k-i)^2 (N-(i+1))}{(i+2)(i+1) \theta} = \frac{(i-k)(i-k)(i+1-N)}{(i+2)(i+1)}\left(-\frac{1}{\theta}\right) .
$$
So, by~\eqref{hyperg}, the factorial moment $M^{\theta}(k,N)=\sum_{i\geq 0} t_i$ has the required hypergeometric representation.
\end{proof}
The first few moments as functions of $k$ are
\begin{align*}
M^{\theta}(k,1)&=\theta ^k\\
M^{\theta}(k,2)&=\frac{1}{2} \theta ^{k-1} \left(k^2+2 \theta\right)\\
M^{\theta}(k,3)&=\frac{1}{6} \theta ^{k-2} \left(k^4-2 k^3+(6 \theta+1)  k^2+6 \theta ^2\right). 
\end{align*}
\par

Consider now the negative binomial distribution with the special parameter $\gamma = 1$, so that the reference measure is simply a geometric distribution $\mathrm{d}\mu_q^1 (x) = q^x (1-q)$. The corresponding measure $Q$ is the Meixner ensemble with parameter $\gamma=1$. 
\begin{thm}[Meixner factorial moment] \label{thm:Meixner}
The factorial moment of the Meixner ensemble with $\gamma=1$ can be written as
\be
M^{1}_q(k,N)
	= \left( \frac{q}{1-q} \right) ^k \frac{(N+1)_k}{k+1}\; {}_3F_2 \left( \begin{matrix} -k , -k , 1-N \\ 1 , -N-k \end{matrix} ; \frac{1}{q} \right) .
	\label{eq:Meix_fact}
\ee
In particular, $\left( (1-q)/q \right) ^{k}\frac{k+1}{(N+1)_k} M^{1}_q(k,N)$ can be extended to the entire complex plane to a polynomial in $k$ of degree $2(N-1)$, and when $k \in \mathbb{N}$, it  is a polynomial in $(N-1)$ of degree $k$.
\end{thm}
\begin{proof}
	When $\gamma=1$, we observe that $\frac{(\gamma+l)_{k-i} l!}{(l-i)!} = \frac{(l+k-i)!}{(l-i)!}$, and by $\eqref{eq:pochsum2}$ the formula of the factorial moment simplifies as
\be
M^{1}_q(k,N)
	= \sum_{i=0}^k  \left( \frac{q}{1-q} \right) ^kq^{-i} \binom{k}{i}^2 \frac{1}{N}\frac{(N-i)_{k+1}}{k+1} .
\ee	
If we denote
$$
t_i=\left( \frac{q}{1-q} \right) ^kq^{-i} \binom{k}{i}^2\frac{1}{N} \frac{(N-i)_{k+1}}{k+1},
$$
then we have 
\begin{align*}
t_0 &= \left( \frac{q}{1-q} \right) ^k\frac{(N+1)_k}{k+1}\\
 \frac{t_{i+1}}{t_i} &= \frac{(k-i)^2(N-i-1)}{(i+1)^2 (N+k-i)q} = \frac{(i-k)(i-k)(i+1-N)}{(i+1)(i-N-k)(i+1)}\left (\frac{1}{q}
 \right),\quad 0\leq i\leq k.
\end{align*}
Using \eqref{hyperg} we conclude the proof.
\end{proof}
The first few moments as functions of $k$  are
\begin{align*}
M^{1}_q(k,1)&=\left(\frac{q}{1-q}\right)^k(2)_{k}\frac{1 }{k+1}\\
M^{1}_q(k,2)&=\left(\frac{q}{1-q}\right)^k(3)_{k}\frac{\left(k^2+k q+2 q\right)}{(k+1) (k+2) q}\\
M^{1}_q(k,3)&=\left(\frac{q}{1-q}\right)^k(4)_{k}\frac{ \left(k^4+4 k^3 q-2 k^3+2 k^2 q^2+8 k^2 q+k^2+10 k q^2+12 q^2\right)}{2 (k+1)(k+2) (k+3) q^2}.
\end{align*}

\subsection{Interpretations in terms of Schur measures}

It is well known that the Charlier and Meixner ensembles can be interpreted in terms of \emph{Schur measures}
on integer partitions.  For a particle configuration $x_1>\cdots>x_N$ on the nonnegative integers, we can 
associate an integer partition $\lambda$ via $\lambda_i=x_i+i-N$.  Under this identification, the Charlier
ensemble with $N$ particles and parameter $\theta$ corresponds to the probability distribution on integer
partitions given by
\be
\C_{\theta,N}(\lambda):= e^{-N\theta} \frac{\theta^{|\lambda|}}{|\lambda|!} s_\lambda(1^N) f^\lambda ,
\ee
as described by Borodin and Olshanski \cite{BO1}, where $s_{\lambda}(x_1, \dots x_N)$ is the Schur polynomial in $N$ variables indexed by $\lambda$, and $f^{\lambda}$ is the number of standard Young tableaux of shape $\lambda$, given by the hook-length formula \cite{Macdonald}.
Furthermore, assuming $\gamma$ is a positive integer, the Meixner ensemble with $N$ particles and parameters 
$\gamma,q$ corresponds to the probability distribution \cite{BO2}
\be
\M_{q,\gamma,N}(\lambda) := q^{|\lambda|} (1-q)^{N(N+\gamma-1)} s_\lambda(1^N) s_\lambda(1^{N+\gamma-1}) .
\ee
For $\lambda$ an integer partition with at most $N$ parts, define
\be
F_{N,k}(\lambda)=\frac1{N}\sum_{i=1}^N [(\lambda_i+N-i) (\lambda_i+N-i-1)\cdots (\lambda_i+N-i-k+1)].
\ee
Then, in the notations of the previous two sections, the factorial moments can be written as follows
\be
M^\theta(k,N)=\sum_\lambda F_{N,k}(\lambda) \C_{\theta,N}(\lambda),
\ee
and 
\be
M^\gamma_q(k,N)=\sum_\lambda F_{N,k}(\lambda) \M_{q,\gamma,N}(\lambda),
\ee
where the sums are taken over all integer partitions $\lambda$.

The fact that $\C_{\theta,N}$ and $\M_{q,\gamma,N}$ are probability measures is a consequence 
of the Cauchy-Littlewood identity~\cite[I.4, (4.3)]{Macdonald}.
Thus the statements of Theorems~1, 2, 4 and 5 can be interpreted as extensions of particular
cases of the Cauchy-Littlewood identities.

\begin{cor}
The following identity relating Schur polynomials and hypergeometric functions holds, as a result of the Charlier factorial moment formula:
\be
\sum_\lambda F_{N,k}(\lambda) e^{-N\theta} \frac{\theta^{|\lambda|}}{|\lambda|!} s_\lambda(1^N) f^\lambda = \theta^k {}_3F_1 \left( \begin{matrix} -k , -k , 1-N \\ 2 \end{matrix} ; - \frac{1}{\theta} \right).
\ee
\end{cor}

\begin{cor}
The following identity relating Schur polynomials and hypergeometric functions holds, as a result of the $\gamma=1$ Meixner factorial moment formula:
\be
\sum_\lambda F_{N,k}(\lambda) q^{|\lambda|} s_\lambda(1^N) s_\lambda(1^N) = \frac{q^k (N+1)_k}{(1-q)^{N^2 + k}(k+1)}\; {}_3F_2 \left( \begin{matrix} -k , -k , 1-N \\ 2, -N-k \end{matrix} ;  \frac{1}{q} \right).
\ee
\end{cor}

\section{Randomised factorial moments}
\label{sec:randomised}
In the classical continuous OP ensembles, the moments are themselves hypergeometric orthogonal polynomials in the variable $k$~\cite{Cunden19}. This is not the case in the discrete setting. However, if we suitably randomise the number of particles $N$, then the factorial moments simplify dramatically and can be expressed as Jacobi polynomials.

Recall that the Jacobi polynomials with shape parameters $(\alpha,\beta)$, are orthogonal with respect to the Beta distribution $(1-x)^{\alpha} (1+x)^{\beta}$ on the interval $x \in [-1,1]$. They have the hypergeometric representation~\cite[9.8]{Koekoek10}
\begin{equation} \label{eq:Jacobi}
P_n^{(\alpha,\beta)}(x) = \frac{(\alpha+1)_n}{n!} {}_2F_1 \left( \begin{matrix}  -n , n+\alpha+\beta+1 \\ \alpha+1 \end{matrix} ; \frac{1-x}{2} \right) . 
\end{equation}
(Here we do not require them to be have squared norm $1$.) When $\alpha=\beta=0$, this is a special case known as the Legendre polynomials, orthogonal with respect to the uniform measure on $x\in[-1,1]$, and we use the standard notation~\cite[9.8.3]{Koekoek10} $ P_n(x)=P_n^{(0,0)}(x)$.

For a given reference measure $\mu$, the OP ensemble $Q$ can be thought of as a random point configuration where the number of points (or particles) is $N$. We now show that the formulae of the factorial moments  in the Charlier ensemble and in the Meixner ensemble with $\gamma=1$ simplify if we randomise $N$.
\subsection{Poissonised Charlier moments}
Consider the OP ensemble $Q$ with reference measure $\mu^{\theta}$, $\theta>0$. Suppose that $N-1$ has Poisson distribution $\mu^{t\theta}$, $t>0$, i.e.
\be
P(N=m+1)=e^{-t\theta}\frac{(t\theta)^m}{m!},\quad m \in \mathbb{N}.
\label{eq:N_Poiss}
\ee
The resulting probability measure is a Poissonised Charlier ensemble.
\begin{thm} \label{th:Poiss}
	Define the Poissonised Charlier factorial moment as
\begin{equation}
M^{\theta}(k;t)	
=\int\left(\int x(x-1) \dots (x-k+1) \mathrm{d} \rho_{m+1}^{\theta}(x)\right)\de\mu^{t\theta}(m).
\end{equation}
Then, $M^{\theta}(k;t)$ can be written in terms of a Jacobi polynomial of degree $k$ in the variable $(1+t)(1-t)^{-1}$, with parameters $\alpha=1$ and $\beta=0$:
\be
M^{\theta}(k;t)=\theta^k {}_2 F_1 \left( \begin{matrix} -k , -k \\ 2 \end{matrix} ; t \right)
	= \frac{\theta^k (1-t)^k}{k+1} P_k^{(1,0)}\left(\frac{1+t}{1-t} \right) .
	\label{eq:mom_Charl_Poi}
\ee
\end{thm}
In the proof we shall make use of the following generating function from the theory of hypergeometric functions.
\begin{lem}[Exton~\cite{Exton98}]
\label{lem:hyper}
	The following identity between hypergeometric functions holds:
\be
 e^{t \theta} {}_2 F_1 \left( \begin{matrix} a , b \\ c \end{matrix} ; t \right) 
	= \sum_{m=0}^{\infty} \frac{(t \theta)^m}{m!} 
	{}_3F_1 \left( \begin{matrix} a , b , -m \\ c \end{matrix}; -\theta^{-1} \right) .
\ee
\end{lem}
\begin{proof}
	We first rewrite the hypergeometric functions and exponential on the left hand side as series, i.e. 
$$
e^{t \theta} {}_2 F_1 \left( \begin{matrix} a , b \\ c \end{matrix} ; t \right) = \sum_{j=0}^{\infty} \frac{(t \theta)^j}{j!} \sum_{k=0}^{\infty} \frac{(a)_k (b)_k}{(c)_k} \frac{t^k}{k!}.
$$
	Then, we change the variables in the double sum by defining $m = k+j$ and eliminating $j$, so we have
	\begin{align*} 
	e^{t \theta} {}_2 F_1 \left( \begin{matrix} a , b \\ c \end{matrix} ; t \right) 
	&= \sum_{m=0}^{\infty} t^m \sum_{k=0}^{m} \frac{(a)_k (b)_k}{(c)_k} \frac{\theta^{m-k}}{(m-k)!k!} \\
	&= \sum_{m=0}^{\infty} \frac{t^m \theta^m}{m!} \sum_{k=0}^m \frac{(a)_k (b)_k}{(c)_k} \frac{m!}{(m-k)!k!} (\theta^{-1}) ^k \\
	&=
	\sum_{m=0}^{\infty} \frac{t^m \theta^m}{m!} \sum_{k=0}^m \frac{(a)_k (b)_k (-m)_k}{(c)_k} \frac{(-\theta^{-1})^k}{k!} ,
		\end{align*}
	which is exactly the form of the hypergeometric function required.
\end{proof}
\begin{proof}[Proof of Theorem~\ref{th:Poiss}]
	We use Lemma~\ref{lem:hyper} with $a=b=-k$ and $c=2$, so that
	\begin{align*} {}_2 F_1 \left( \begin{matrix} -k , -k \\ 2 \end{matrix} ; t \right) 
	&= \sum_{m=0}^{\infty} \frac{e^{-t \theta} (t \theta)^m}{m!} {}_3 F_1 \left( \begin{matrix} -k , -k , -m \\ 2 \end{matrix} ; - \theta^{-1} \right) \\
	&= \int {}_3 F_1 \left( \begin{matrix} -k , -k , -m \\ 2 \end{matrix} ; - \theta^{-1} \right) \mathrm{d} \mu^{t \theta}(m), 
	\end{align*}
by definition of the Poisson measure $\mu^{t \theta}$. Hence by the factorial moment formula of Theorem \ref{thm:Charlier}, we have
$$
M^{\theta}(k;t)=	\theta^k {}_2 F_1 \left( \begin{matrix} -k , -k \\ 2 \end{matrix} ; t \right).
$$
To show the second equality~\eqref{eq:mom_Charl_Poi} we use the hypergeometric representation~\eqref{eq:Jacobi} of the Jacobi polynomial. \end{proof} 
The first few Poissonised moments are
\begin{align*}
M^{\theta}(1;t)&=\frac{\theta}{2}   (t+2) &
M^{\theta}(2;t)&=\frac{\theta ^2}{3}  \left(t^2+6 t+3\right)\\
M^{\theta}(3;t)&=\frac{\theta ^3}{4}  \left(t^3+12 t^2+18 t+4\right) &
M^{\theta}(4;t)&=\frac{ \theta ^4}{5} \left(t^4+20 t^3+60 t^2+40 t+5\right).
\end{align*}
Note that~\eqref{eq:mom_Charl_Poi} are polynomials in $t$. In particular, they are regular at $t=1$, corresponding to the Poissonised Charlier ensemble with parameter $\theta$.
\begin{cor} 
    When we set $t=1$ (i.e. Poissonise with parameter $\theta$), we have
\be
M^{\theta}(k;1) = \frac{\theta^k}{k+1} \binom{2k+1}{k} = \theta^k  \frac{2k+1}{k+1}C_k, 
 \ee
	where $C_k=\frac{1}{k+1}\binom{2k}{k}$ is the $k$th Catalan number.
\end{cor}

\begin{proof}
The second equality is clear. To show the first equality, we observe that
$$
 {}_2 F_1 \left( \begin{matrix} -k , -k \\ 2 \end{matrix} ; 1 \right) = \sum_{j=0}^k \frac{1}{j+1} \binom{k}{j}^2 = \frac{1}{k+1} \sum_{j=0}^k \binom{k+1}{k-j} \binom{k}{j}. 
 $$
Then, by the Chu-Vandermonde identity, we have the binomial term as required.
\end{proof}

The Jacobi polynomials satisfy a three-term recurrence~\cite[Eq.~(9.8.4)]{Koekoek10} (in the degree) and a second order differential equation~\cite[Eq.~(9.8.6)]{Koekoek10} (in the argument). In the special case $(\alpha,\beta)=(1,0)$ they are
\be
xP_{k}^{(1,0)}(x) = \frac{k+2}{2k+3}P_{k+1}^{(1,0)}(x) 
- \frac{1}{(2k+1)(2k+3)}P_{k}^{(1,0)}(x) 
+\frac{k}{2k+1}P_{k-1}^{(1,0)}(x),
\label{eq:rec_Jac}
\ee
\be
(1-x^2) \frac{\mathrm{d}^2 }{\mathrm{d}x^2} P_{k}^{(1,0)}(x)- (1+3x) \frac{\mathrm{d}}{\mathrm{d}x}P_{k}^{(1,0)}(x) + k(k+2)P_{k}^{(1,0)}(x) = 0.
\label{eq:ode_Jac}
\ee
These, combined with Theorem~\ref{th:Poiss}, show that the Poissonised factorial moments $M^{\theta}(k,t)$ of the Charlier ensemble satisfy a three-term recurrences in $k$ and a second order differential equation in $t$.
\begin{cor}
The Poissonised moments $M^{\theta}(k,t)$ satisfy the  three-term recurrence
\begin{multline} 
(2k+1)(k+2)^2M^{\theta}(k+1,t) =  \theta(k+1) \left[(2k+3)(2k+1)(1+t)+(1-t) \right] M^{\theta}(k,t) \\
-  \theta^2 k^2(2k+3)(1-t)^2  M^{\theta}(k-1,t),
\label{eq:rec_Charl_Poi}
 \end{multline}
and the second order differential equation
\be
\frac{t}{(1-t)^{k-2}}\frac{\mathrm{d}^2 }{\mathrm{d}t^2}M^{\theta}(k,t)  +2^{k-2}(2k+3)\frac{\mathrm{d} }{\mathrm{d}t} M^{\theta}(k,t) +2^{k-1} k^2M^{\theta}(k,t)  =0.
\label{eq:ode_Charl_Poi}
\ee
\end{cor}
\begin{proof} The recursion~\eqref{eq:rec_Charl_Poi} is an immediate consequence of~\eqref{eq:rec_Jac}. 
For the proof of the differential equation~\eqref{eq:ode_Charl_Poi}, note that from~\eqref{eq:ode_Jac}  we have
\be
t(1-t)^2\frac{\mathrm{d}^2 }{\mathrm{d}t^2} P_{k}^{(1,0)}\left(\frac{1+t}{1-t}\right) +(2-t)(1-t)\frac{\mathrm{d} }{\mathrm{d}t} P_{k}^{(1,0)}\left(\frac{1+t}{1-t}\right) -k(k+2) P_{k}^{(1,0)}\left(\frac{1+t}{1-t}\right)=0.
\ee
A straightforward manipulation concludes the proof.
\end{proof}

\subsection{Negative binomialised Meixner moments}
For the Meixner ensemble with \hbox{$\gamma=1$} it turns out that negative binomialising the number of particles provides tractable formulae. More precisely, consider the OP ensemble $Q$ with geometric reference measure $\mu^{1}_q$, $0<q<1$. Suppose that $N-1$ has $2$-negative binomial distribution $\mu^{2}_{tq}$, $t>0$, i.e. 
\be
P(N=m+1)= (m+1)(tq)^m (1-tq)^{2},\quad m \in \mathbb{N}.
\label{eq:N_Negbin}
\ee

\begin{thm}
\label{thm:Meix}
Define the 2-negative binomialised Meixner factorial moment as
\begin{equation}
M^{1,2}_q(k;t)
=\int\left(\int x(x-1) \dots (x-k+1) \mathrm{d} \rho^{1}_{m+1,q}(x)\right)\de\mu^{2}_{tq}(m).
\end{equation}
Then, $M^{1,2}_q(k;t)$ can be written in terms of a a Legendre polynomial in the variable $(1+t)(1-t)^{-1}$:
\begin{multline}
M^{1,2}_q(k;t)= k!\left( \frac{q}{(1-q)(1-tq)} \right) ^k  {}_2F_1 \left( \begin{matrix} -k , -k \\ 1 \end{matrix} ; t \right) \\
	= k!\left( \frac{q}{(1-q)(1-tq)} \right)^k (1-t)^k P_k \left( \frac{1+t}{1-t} \right).
	\label{eq:mom_Meix_Negbin}
\end{multline}
\end{thm}

\begin{proof}
	We first follow a similar method to Exton's lemma to find the hypergeometric representation in~\eqref{eq:mom_Meix_Negbin}, and then compare with~\eqref{eq:Jacobi} to show that this is a Legendre polynomial.

By~\eqref{eq:Meix_fact} and the definition of $\mu^{2}_{tq}$, we have
\begin{align*}
M^{1,2}_q(k;t)&=\int \left(\sum_{i=0}^k  \left( \frac{q}{1-q} \right) ^kq^{-i} \binom{k}{i}^2 \frac{1}{m+1}\frac{(m+1-i)_{k+1}}{k+1}\right)\de\mu^{2}_{tq}(m) \\
&= \frac{1}{k+1}\left( \frac{q}{1-q} \right) ^k\sum_{m=0}^{\infty}(m+1)(tq)^m (1-tq)^{2} \sum_{i=0}^k q^{-i} \binom{k}{i}^2\frac{(m+1-i)_{k+1}}{m+1}\\
&=\frac{1}{k+1}\left( \frac{q}{1-q} \right) ^k(1-tq)^2 \sum_{l=0}^{\infty} \sum_{i=0}^k t^l q^l (l+1)_{k+1} t^i \binom{k}{i}^2 \\
&=\frac{1}{k+1}\left( \frac{q}{1-q} \right) ^k(1-tq)^2  \left(\sum_{l=0}^{\infty} t^l q^l (l+1)_{k+1}\right)\cdot \left(\sum_{i=0}^{k} t^i \binom{k}{i}^2\right)\\
&=\frac{1}{k+1}\left( \frac{q}{1-q} \right) ^k(1-tq)^2 \left( \frac{ (k+1)! }{(1-tq)^{k+2}}\right)\cdot \left({}_2F_1 \left( \begin{matrix} -k , -k \\ 1 \end{matrix} ; t \right)\right),
\end{align*}
where in the third line we relabelled $l=m-i$. 
Simplifying the terms we get the first equality in~\eqref{eq:mom_Meix_Negbin}.
	By \eqref{eq:Jacobi}, 
	$$
	P_k \left( \frac{t+1}{t-1} \right) = (t-1)^{-k} {}_2F_1 \left( \begin{matrix} -k , -k \\ 1 \end{matrix} ; t \right),
	$$
and the second equality in~\eqref{eq:mom_Meix_Negbin} follows.
\end{proof}

\begin{cor}
    When we set $t=1$, i.e. negative binomialise with parameter $q$, we have
\be
M^{1,2}_q(k;1)=\frac{q^k}{(1-q)^{2k}} (k+1)! C_k. 
\ee
\end{cor}

\begin{proof}
An immediate consequence of
\[
 {}_2F_1 \left( \begin{matrix} -k , -k \\ 1 \end{matrix} ; 1 \right) = \sum_{j=0}^k \binom{k}{j}^2 = \binom{2k}{k}.
\]
\end{proof}
\par
As in the Charlier case, the negative binomialised factorial moments of the $\gamma=1$ Meixner ensemble satisfy three term recurrences in $k$ and differential equations in the parameter $t$.
\begin{cor} The negative binomialised factorial moments $M^{1,2}_q(k;t)$ satisfy the three term recurrences
\be
M^{1,2}_q(k+1;t) =  \frac{(2k+1) q(1+t)}{(1-q)(1-tq)}  M^{1,2}_q(k;t) 
-  \left(\frac{kq(1-t)}{(1-q)(1-tq)} \right)^2 M^{1,2}_q(k-1;t),
\ee
and the second order differential equation
\begin{multline}
    t(1-t)\frac{\mathrm{d}^2 }{\mathrm{d}t^2}M^{1,2}_q(k;t)+ \frac{qt^2 - (1-2k(1-q)+q)t+1}{1-qt}\frac{\mathrm{d} }{\mathrm{d}t}M^{1,2}_q(k;t) \\
    - \frac{kq(1-t)+k^2(1-q^2t )}{(1-qt)^2}M^{1,2}_q(k;t) = 0.
    \end{multline}
\end{cor}

\begin{proof}
A consequence of the three-term recursion~\cite[Eq.~(9.8.64)]{Koekoek10} in $k$ and the differential equation~\cite[Eq.~(9.8.66)]{Koekoek10} in $x$ of the Legendre polynomials $P_k(x)$,
\be
xP_k(x) = (k+1)P_{k+1}(x)-2kP_k(x)+kP_{k-1}(x),
\label{eq:rec_Leg}
 \ee
\be
(1-x)^2 \frac{\mathrm{d}^2 }{\mathrm{d}x^2} P_k(x)- 2x\frac{\mathrm{d}}{\mathrm{d}x} P_k(x)+ k(k+1)P_k(x) = 0.
\label{eq:ode_Leg}
\ee
\end{proof}

\begin{rem}
In place of the Legendre polynomials, we could also use the Legendre rational functions $R_k(x)$, which are defined by
$$
R_k(x) = \frac{\sqrt{2}}{x+1} P_k \left( \frac{x-1}{x+1} \right),
$$
and form an orthogonal system of rational functions on the positive half-line with respect to the weight $\omega(x)=(1+x)^{-2}$. See~\cite{Guo00}. 
\end{rem}

We conclude this section by remarking that the above results may also be interpreted in terms
of the Schur measures described in Section 3.2.

\begin{cor}
The Poissonised Charlier factorial moment can be written as the following Cauchy-like weighted sum over Schur polynomials:
\be
\sum_{N=1}^{\infty} \sum_\lambda F_{N,k}(\lambda) e^{-(N+t)\theta} \frac{ \theta^{|\lambda|}}{|\lambda|!} \frac{(t \theta)^{N-1}}{(N-1)!} s_\lambda(1^N) f^\lambda =
\theta^k {}_2 F_1 \left( \begin{matrix} -k , -k \\ 2 \end{matrix} ; t \right).
\ee
\end{cor}

\begin{cor}
The 2-Negative Binomialised $(\gamma=1)$ Meixner factorial moment can be written as the following Cauchy-like weighted sum over Schur polynomials:
\be
\sum_{N=1}^{\infty} \sum_\lambda N F_{N,k}(\lambda) q^{|\lambda|} (tq)^{N-1} s_\lambda(1^N)^2 =
 \frac{q^k k!}{(1-q)^k(1-tq)^{k+2}}  {}_2F_1 \left( \begin{matrix} -k , -k \\ 1 \end{matrix} ; t \right). 
\ee
\end{cor}

\section{Equilibrium measures}
Let $X=(X_1,\dots,X_N)$ be distributed according to an OP ensemble $Q$. The \emph{mean spectral measure} is the probability measure $\widetilde{\rho}_N=\EE_Q \frac{1}{N} \sum_{n=1}^N\delta_{X_n/N}$, and it is a simple rescaling of the one-point function
\be
d\widetilde{\rho}_N(x)=d\rho_N(Nx)=\frac{1}{N}\sum_{n=0}^{N-1}p_n(Nx)^2\de\mu(Nx).
\ee
It is well-known that in the limit $N\to\infty$, under some quite general conditions, the mean spectral measure of OP ensembles converges to a limiting probability measure known as \emph{equilibrium measure} (it is the unique minimiser of a certain energy functional). For the classical OP ensembles (continuous and discrete), the equilibrium measure is known explicitly in terms of elementary functions~\cite{Johansson99,Johansson00,Johansson02,Kuijlaars99,Dragnev97,Saff97}. In the context of the 
corresponding Schur measures, these equilibrium measures describe limit shapes of random partitions and have close
connections to asymptotic representation theory and free probability~\cite{Biane}.

In~\cite{Ledoux04,Ledoux05}, Ledoux gave a precise description of the equilibrium measures of the classical ensembles as the  distribution of adapted mixtures of an arcsine random variable with an independent uniform random variable. 

In the following, $\xi$ will be a random variable with the arcsine distribution on $(-1,1)$, and $U$ a uniform random variable on $(0,1)$, independent of $\xi$, so that
\be
P(\xi\leq s, U\leq u)=\int_{-\infty}^{s} \frac{\chi_{(-1,1)}(s')}{\pi \sqrt{1-s'^2}}\de s'\int_{-\infty}^{u} \chi_{(0,1)}(u')\de u'.
\ee
\subsection{Charlier ensemble} The equilibrium measure can be described as follows.
\begin{thm}[Ledoux~\cite{Ledoux05}]\label{thm:Charlier_eqmeas} Consider the mean spectral measure of the Charlier ensemble $\widetilde{\rho}_N^{\theta_N}$, and suppose that $\theta_N \sim hN$ as $N \to \infty$. Then, as $N \to \infty$, $\widetilde{\rho}_N^{\theta_N}$ converges weakly to the distribution of the random variable
\be
U+  2 \sqrt{Uh} \xi + h.
\label{eq:Led_1}
 \ee
 
\end{thm}
We now explain how the Poissonised Charlier moment formulae relate to Theorem~\ref{thm:Charlier_eqmeas}. 
When we Poissonise the Charlier ensemble, we take $(N-1)$ a Poisson random variable with rate $t\theta$, see~\eqref{eq:N_Poiss}. Therefore, $\EE N=1+\theta t$ and $\operatorname{Var} (N)=\theta t$, and we see that  $N/\theta \to t$ in probability, as $\theta\to\infty$. In other words, $\theta \sim N/t$, which is the assumption needed for Ledoux's theorem above, with $h = 1/t$.

\begin{lem}
The $k$th moment of the random variable $ U+2 \sqrt{Uh} \xi + h$ is
\be
 \int_0^1 \int_{-1}^1 \left(u+2 \sqrt{uh} s +h\right)^k \frac{\mathrm{d} s \mathrm{d} u}{\pi \sqrt{1-s^2}} = {h^k} {}_2F_1 \left( \begin{matrix} -k , -k \\ 2 \end{matrix} ; \frac{1}{h} \right) .
 \label{eq:mom_Charl_eqmeas}
 \ee
\end{lem}
\begin{proof}
First we substitute $\sin x = s$ and then use a multinomial expansion to get
$$
\frac{1}{\pi}\int_0^1 \int_{-\frac{\pi}{2}}^{\frac{\pi}{2}} \sum_{j=0}^k \sum_{l=0}^{k-j} \frac{k!}{j!l!(k-j-l)!} 2^j u^{\frac{j}{2}} h^{\frac{j}{2}} (\sin x)^j u^l h^{k-j-l} \mathrm{d}x \mathrm{d} u. 
$$
For the integral in $x$ we use
\begin{equation} \label{eq:sin}
 \int_{-\frac{\pi}{2}}^{\frac{\pi}{2}} \sin^j{(x)} \mathrm{d}x = \begin{cases}
   \frac{\pi}{2^j} \binom{j}{\frac{j}{2}}   & \quad \text{if } j \text{ is even}\\
    0  & \quad \text{if } j \text{ is odd}
\end{cases}  ,
\end{equation}
so that
\[
 \int_0^1 \int_{-1}^1 \left(u+2 \sqrt{uh} s +h\right)^k \frac{\mathrm{d} s \mathrm{d} u}{\pi \sqrt{1-s^2}} 
= h^k \sum_{j=0}^{\lfloor \frac{k}{2} \rfloor} \sum_{l=0}^{k-2j} \frac{k!}{(j!)^2 l! (k-2j-l)!} \frac{h^{-j-l}}{(j+l+1)}, 
\]
so it just remains to show that the double sum is equal to ${}_2F_1 \left( \begin{matrix} -k , -k \\ 2 \end{matrix} ; \frac{1}{h} \right)$ to complete the proof.

To do so, we first note that this hypergeometric function can be written as
\[
{}_2F_1 \left( \begin{matrix} -k , -k \\ 2 \end{matrix} ; \frac{1}{h} \right) = \sum_{m=0}^k \binom{k}{m}^2 \frac{h^{-m}}{m+1} = \sum_{m=0}^k \binom{k}{m} \binom{k}{k-m} \frac{h^{-m}}{m+1}, 
\]
and using an identity from Riordan \cite[Chapter 1, Problem 2(a)]{Riordan} we have
\[ \binom{k}{m} \binom{k}{k-m} = \sum_{j=0}^{m \wedge (k-m)} \binom{m}{j} \binom{k-j}{k-m-j} \binom{k}{k-j} = \sum_{j=0}^{\lfloor \frac{k}{2} \rfloor} \binom{m}{j} \binom{k-j}{k-m-j} \binom{k}{k-j}, \]
since clearly we have $0 \leq j \leq \lfloor \frac{k}{2} \rfloor$ in this sum. Thus, substituting $l = m-j$ and swapping the order of summation, we have
\[ {}_2F_1 \left( \begin{matrix} -k , -k \\ 2 \end{matrix} ; \frac{1}{h} \right) =\sum_{j=0}^{\lfloor \frac{k}{2} \rfloor} \sum_{l=0}^{k-2j} \binom{j+l}{j} \binom{k-j}{k-2j-l} \binom{k}{k-j} \frac{h^{-j-l}}{j+l+1} \]
where $k-2j$ is the maximum value of $l$ that keeps each binomial coefficient non-zero. Clearly the binomial coefficients simplify to give the double sum above, so the proof is complete.

\end{proof}

Comparing the moments~\eqref{eq:mom_Charl_eqmeas} of the equilibrium measure with Theorem~\ref{th:Poiss} we readily get the following result.
\begin{cor}\label{cor:Charl_eq} For all $\theta,h>0$, the $k$th Poissonised factorial moments are related to the moments of~\eqref{eq:Led_1}  in the following way:
\be
 \int_0^1 \int_{-1}^1 \left(u+2 \sqrt{uh} s +h\right)^k \frac{\mathrm{d} s \mathrm{d} u}{\pi \sqrt{1-s^2}} =  \frac{M^{\theta}(k;1/h)}{(\theta/h)^k}.
\ee
\end{cor}
After rescaling, the factorial moments converge to the moments as $\theta \to \infty$. Therefore, the moments of the Poissonised Charlier mean spectral measure converge to the moments of $(U+2 \sqrt{Uh} \xi +h)$, which is consistent with Theorem~\ref{thm:Charlier_eqmeas}.

\subsection{Meixner ensemble} The equilibrium measure for the Meixner ensemble can be also expressed in terms of independent arcsine and uniform random variables. 
\begin{thm}[Ledoux~\cite{Ledoux05}]\label{xx}
 Consider the mean spectral measure of the Meixer ensemble $\widetilde{\rho}^{\gamma_N}_{N,q}$, and suppose that $\gamma_N \sim cN$ as $N \to \infty$. Then, as $N \to \infty$, $\widetilde{\rho}^{\gamma_N}_{N,q}$ converges weakly to the distribution of the random variable
\be
 \frac{1}{1-q} \left(U+ 2 \sqrt{qU(c+U)}\xi  + q(c+U) \right).
 \ee
\end{thm}
Here we consider the special case $\gamma=1$, that is $c=0$ (in fact for any fixed $\gamma$ this is the case), when the limiting random variable takes the simpler form of $U(1+q+2\sqrt{q} \xi)(1-q)^{-1}$. As in the Charlier case, we now calculate the $k$th moment of this random variable.
\begin{lem}\label{lem:Meix_mom_eqmeas}
The $k$th moment of the random variable $ \frac{1}{1-q}(U(1+q+2\sqrt{q} \xi))$ is given by
\be
 \iint \left( \frac{u(1+q+2\sqrt{q} s)}{1-q} \right)^k \frac{\mathrm{d} s \mathrm{d} u}{\pi \sqrt{1-s^2}}  =  \frac{1}{k+1} \left( \frac{q}{1-q} \right)^k {}_2F_1 \left( \begin{matrix} -k , -k \\ 1 \end{matrix} ; \frac{1}{q} \right)  .
 \ee
\end{lem}
\begin{proof}
The $k$th moment is the product
$$ \int_0^1 \left( \frac{u}{1-q} \right)^k \mathrm{d} u \int_{-1}^1 \frac{(1+q+2 \sqrt{q} s)^k}{\pi \sqrt{1-s^2}} \mathrm{d} s, 
$$
where the first integral is $(1-q)^{-k}(k+1)^{-1}$, and the second can be transformed by a substitution and binomial expansion to
$$ \frac{1}{\pi} \int_{-\frac{\pi}{2}}^{\frac{\pi}{2}} (1+q+2 \sqrt{q} \sin{x})^k \mathrm{d}x =
\frac{1}{\pi} \sum_{j=0}^k \binom{k}{j} (1+q)^{k-j} 2^j q^{\frac{j}{2}} \int_{-\frac{\pi}{2}}^{\frac{\pi}{2}} \sin^j{(x)} \mathrm{d}x . 
$$
We can once again use \eqref{eq:sin} to see that
\begin{align*} \frac{1}{\pi} \int_{-\frac{\pi}{2}}^{\frac{\pi}{2}} (1+q+2 \sqrt{q} \sin{x})^k \mathrm{d}x &= 
\sum_{j=0, j \text{ even }}^k \frac{k!}{(k-j)! (\frac{j}{2})!(\frac{j}{2})!}
(1+q)^{k-j} q^{\frac{j}{2}} \\
&=  (1+q)^k\sum_{j=0}^{\lfloor \frac{k}{2} \rfloor} \frac{k!}{(k-2j)!(j!)^2}
(1+q)^{-2j} q^j, 
\end{align*}
and the sum is equal to $ \left( \frac{q}{1+q} \right)^k {}_2F_1 \left( \begin{matrix} -k , -k \\ 1 \end{matrix} ; \frac{1}{q} \right)$.
\end{proof}

In the Meixner case, the connection between factorial moments and moments of the equilibrium measure is not as straightforward as in the Charlier ensemble. Nevertheless, there is a natural way to obtain a match via a $2$-negative binomialised $(N-1)$, weighting the factorial moments, and then taking the limit $t\to1/q$. We state this as a corollary of Theorem~\ref{thm:Meixner}.

\begin{cor}\label{cor:Meix_mom_eqmeas}
Consider the Meixner ensemble with parameters $\gamma=1$ and $0<q<1$, and $k$th factorial moments $M^{1}_q(k,N)$. Suppose that $(N-1)$ has $2$-negative binomial distribution~\eqref{eq:N_Negbin}. Then,
\be
 \iint \left( \frac{u(1+q+2\sqrt{q} s)}{1-q} \right)^k \frac{\mathrm{d} s \mathrm{d} u}{\pi \sqrt{1-s^2}}  =\lim_{t \to \frac{1}{q}} \int \frac{M^{1}_q(k,m+1)}{(m+2)_k} \mathrm{d} \mu_{tq}^2(m) .
 \ee
\end{cor}
\begin{proof} From the explicit formula~\eqref{eq:Meix_fact} we get
\begin{align*} _{}
 \int \frac{M^{1}_q(k,m+1)}{(m+2)_k} \mathrm{d} \mu_{tq}^2(m) 
&= \sum_{m=0}^{\infty} \sum_{j=0}^k \binom{k}{j}^2 \frac{(-m)_j}{(-m-k-1)_j} (tq)^m (1-tq)^2 (m+1) q^{-j} \\
&=  \sum_{j=0}^k \binom{k}{j}^2 t^j (1-tq)^2 \sum_{l=0}^{\infty} \frac{(-l-j)_j}{(-l-j-k-1)_j} (l+j+1) (tq)^l \\
&=  \sum_{j=0}^k \binom{k}{j}^2 t^j (1-tq)^2 \sum_{l=0}^{\infty} \frac{(l+2)_j}{(l+k+2)_j} (l+1) (tq)^l.
\end{align*}
The $l$ sum is 
\begin{align*} \sum_{l=0}^{\infty} \frac{(l+2)_j (2)_l}{(l+k+2)_j l!} (tq)^l
= \frac{(j+1)! (k+1)!}{(j+k+1)!} {}_2F_1 \left( \begin{matrix} k+2 , j+2  \\ k+j+2 \end{matrix} ; tq \right).
\end{align*}
By Euler's hypergeometric transformation,
\begin{equation*}
(1-tq)^2 {}_2F_1 \left( \begin{matrix} k+2 , j+2  \\ k+j+2 \end{matrix} ; tq \right) = {}_2F_1 \left( \begin{matrix} j , k  \\ k+j+2 \end{matrix} ; tq \right),
\end{equation*}
and by Gauss' hypergeometric theorem, the function on the right hand side is equal to $\frac{(j+k+1)!}{(j+1)!(k+1)!}$ if we take $tq \to 1$.
We therefore have the limit
$$
 \lim_{t \to \frac{1}{q}}(1-tq)^2 {}_2F_1 \left( \begin{matrix} k+2 , j+2  \\ k+j+2 \end{matrix} ; tq \right) = \frac{(j+k+1)!}{(j+1)!(k+1)!} . 
 $$
The factorial terms cancel and we get
\[
\frac{1}{k+1} \left( \frac{q}{1-q} \right)^k \sum_{j=0}^k \binom{k}{j}^2 \frac{1}{q^j}=\frac{1}{k+1} \left( \frac{q}{1-q} \right)^k {}_2F_1 \left( \begin{matrix} -k , -k \\ 1 \end{matrix} ; \frac{1}{q} \right), 
\]
which is exactly the hypergeometric representation of Lemma~\ref{lem:Meix_mom_eqmeas}.
\end{proof}

The above corollary is consistent with Theorem~\ref{xx}, which gives
the equilibrium measure for the Meixner ensemble with $\gamma=1$ (i.e. $c=0$). 
Note that, in this scaling limit, the 2-negative binomial with parameter $tq$,
multiplied by $(1-tq)$, converges in law as $t\to 1/q$,
to a standard gamma random variable with parameter $2$.

\begin{rem}
When $c > 0$, the moments of the limiting distribution become
\be
\sum_{j=0}^{\lfloor \frac{k}{2} \rfloor} \sum_{l=0}^{k-2j} \sum_{m=0}^j \frac{k! c^{l+m}}{l! j! m! (j-m)! (k-2j-l)!(k-m-l+1)}\frac{q^{j+l}}{(1+q)^{2j+l}}. 
\ee
We have been unable to find a hypergeometric representation for this. If it were possible, it would suggest how (if possible at all) to randomise $N$ in a way to obtain tractable results and recover the equilibrium measure for general $\gamma\sim cN$.

Similarly, in the Krawtchouk ensemble for $K \sim \kappa N$, with $\kappa>1$ and $0<p<1$, the mean spectral measure converges to the distribution of~\cite{Ledoux05} 
\be
(1-p)U+2 \sqrt{p(1-p)U(\kappa-U)} \xi + p(\kappa-U).
\label{eq:Kraw_mom_eqmeas}
\ee
We could calculate the moments of~\eqref{eq:Kraw_mom_eqmeas} as 
\be
 \sum_{j=0}^{\lfloor \frac{k}{2} \rfloor} \sum_{l=0}^{k-2j} \sum_{m=0}^j 
\frac{p^kk!(-1)^m \kappa^{k-j-l-m}}{l! j! m! (j-m)! (k-2j-l)!(j+l+m+1)} \left( \frac{1-p}{p} \right)^j \left( \frac{1-2p}{p} \right)^l, \ee
but, as in the general Meixner case, we have not been able to properly identify a useful hypergeometric representation. 
\end{rem}

  \end{document}